\documentclass[graybox]{svmult}

\usepackage{type1cm}        
\usepackage{makeidx}         
\usepackage{graphicx}        
\usepackage{multicol}        
\usepackage[bottom]{footmisc}

\usepackage{newtxtext}       %
\usepackage[varvw]{newtxmath}       
\usepackage{hyperref}
\usepackage{cleveref}
\usepackage{cite}
\usepackage[percent]{overpic}

\makeindex             

\usepackage{epstopdf}
\usepackage[percent]{overpic}
\usepackage{tikz}
\usetikzlibrary{matrix}
\usepackage{algorithm}
\usepackage{algorithmic}

\begin{document}

\title*{\texttt{SpecSolve}: Spectral methods for spectral measures}
\author{Matthew J. Colbrook, Andrew Horning}
\institute{Matthew J. Colbrook \at Department of Applied Mathematics and Theoretical Physics, University of Cambridge,
Cambridge, CB3 0WA and Centre Sciences des Données, Ecole Normale Supérieure, 45 rue d'Ulm, 75005 Paris.  \email{m.colbrook@damtp.cam.ac.uk}
\and Andrew Horning \at Department of Mathematics, Massachusetts Institute of Technology, 182 Memorial Dr, Cambridge, MA 02142, United States. \email{horninga@mit.edu}}
\maketitle

\abstract{Self-adjoint operators on infinite-dimensional spaces with continuous spectra are abundant but do not possess a basis of eigenfunctions. Rather, diagonalization is achieved through spectral measures. The \texttt{SpecSolve} package [SIAM Rev., 63(3) (2021), pp. 489--524]  computes spectral measures of general (self-adjoint) differential and integral operators by combining state-of-the-art adaptive spectral methods with an efficient resolvent-based strategy. The algorithm achieves arbitrarily high orders of convergence in terms of a smoothing parameter, allowing computation of both discrete and continuous spectral components. This article extends \texttt{SpecSolve} to two important classes of operators: singular integro-differential operators and general operator pencils. Essential computational steps are performed with off-the-shelf spectral methods, including spectral methods on the real line, the ultraspherical spectral method, Chebyshev and Fourier spectral methods, and the ($hp$-adaptive and sparse) ultraspherical spectral element method. This collection illustrates the power and flexibility of \texttt{SpecSolve}'s ``discretization-oblivious'' paradigm.}

\vspace{2mm}



\section{Introduction}

Any finite and self-adjoint matrix $A\in\mathbb{C}^{n\times n}$ has an orthonormal basis of eigenfunctions. This basis diagonalizes $A$ by decomposing the space $\mathbb{C}^n$ into a sum of orthogonal eigenspaces. However, many applications require us to study a self-adjoint operator $\mathcal{L}$ with domain $\mathcal{D}(\mathcal{L})\subset\mathcal{H}$ on an \textit{infinite-dimensional} Hilbert space $\mathcal{H}$ with inner product $\langle \cdot,\cdot \rangle$. Even when given a finite matrix $A$, it is often an approximation or discretization of an underlying infinite-dimensional operator. In infinite dimensions, there may not exist a basis of eigenfunctions since $\mathcal{L}$ can have a continuous spectral component. This situation arises in, for example, stochastic processes and signal-processing \cite{kallianpur1971spectral,girardin2003semigroup}~\cite[Ch.~7]{rosenblatt1991stochastic}, scattering in particle physics \cite{efros2007lorentz,efros1994response}, density-of-states in materials \cite{haydock1972electronic,lin2016approximating}, and many other areas \cite{wilkening2015spectral,killip2003sum,MR838253,trogdon2012numerical}.

Instead of eigenfunctions, $\mathcal{L}$ can be diagonalized through spectral measures supported on its spectrum $\Lambda(\mathcal{L})\subset\mathbb{R}$ (see \cref{mcolb_sec:spec_meas} and \cref{mcolb_eqn:cts_decomp}). While efficient methods for computing spectral measures of (even very large) finite matrices exist \cite{lin2016approximating}, the infinite-dimensional case is more subtle. Most existing methods focus on specific operators where analytical formulas are available or perturbations of such cases \cite[Section 3]{colbrook2020}. Recently, \cite{colbrook2020} developed methods for computing spectral measures of general ODEs and integral operators using two ingredients:
\begin{enumerate}
	\item A numerical solver for shifted linear equations $(\mathcal{L}-z)u=f$ for $z\in\mathbb{C}\backslash\Lambda(\mathcal{L})$.
	\item Numerical approximations to inner products of the form $\langle u,f \rangle.$
\end{enumerate}
The software \texttt{SpecSolve} \cite{SpecSolve_code} implements these ingredients using spectral methods.

This article extends \texttt{SpecSolve} to two important classes of operators with continuous spectra: singular integro-differential operators and operator pencils. Leveraging sparse spectral methods for the Hilbert transform on the real line, we compute spectral measures of singular integral operators such as
\begin{equation}\label{mcolb_eqn:SIO}\setlength\abovedisplayskip{6pt}\setlength\belowdisplayskip{6pt}
[\mathcal{L}u](x)=a(x)u(x) + \frac{1}{\pi i}\int_\mathbb{R}\frac{G(x,y)}{y-x}\,u(y)\,dy,
\end{equation}
where $G(x,y)=\overline{G(y,x)}$ and real-valued $a(x)$ satisfy appropriate regularity constraints on $\mathbb{R}$. Differential terms are straightforward to incorporate to tackle a broad class of singular integro-differential operators. We also extend the two-step framework to compute spectral measures associated with the generalized spectral problem $\mathcal{A}v=\lambda \mathcal{B}v,$ for operators $\mathcal{A}$ and $\mathcal{B}$. The two essential computational steps are performed with off-the-shelf spectral methods, illustrating the power and flexibility of \texttt{SpecSolve}'s ``discretization-oblivious" paradigm.

\section{Spectral measures}
\label{mcolb_sec:spec_meas}

The spectral theorem for a finite self-adjoint matrix $A\in \mathbb{C}^{n\times n}$ states that there exists an orthonormal basis of eigenvectors $v_1,\dots,v_n$ for $\mathbb{C}^n$ such that
\begin{equation}\label{mcolb_eqn:disc_decomp}\setlength\abovedisplayskip{6pt}\setlength\belowdisplayskip{6pt}
v = \left(\sum_{k=1}^n v_kv_k^*\right)v, \quad v\in\mathbb{C}^n \qquad\text{and}\qquad Av = \left(\sum_{k=1}^n\lambda_k v_kv_k^*\right)v, \quad v\in\mathbb{C}^n,
\end{equation}
where $\lambda_1,\ldots,\lambda_n$ are eigenvalues of $A$, i.e., $Av_k = \lambda_kv_k$ for $1\leq k\leq n$. In other words, the projections $v_kv_k^*$ decompose $\mathbb{C}^n$ and diagonalize $A$.

Switching to infinite dimensions, associated with the operator $\mathcal{L}$ is a projection-valued measure, $\mathcal{E}$ \cite[Theorem VIII.6]{reed1972methods}, whose support is the spectrum $\Lambda(\mathcal{L})$. The measure $\mathcal{E}$ assigns an orthogonal projector to each Borel subset of $\mathbb{R}$ such that
\begin{equation}\label{mcolb_eqn:cts_decomp}\setlength\abovedisplayskip{6pt}\setlength\belowdisplayskip{6pt}
	f=\left(\int_\mathbb{R} d\mathcal{E}(y)\right)f,\quad f\in\mathcal{H} \qquad\text{and}\qquad \mathcal{L}f=\left(\int_\mathbb{R} y\,d\mathcal{E}(y)\right)f, \quad f\in\mathcal{D}(\mathcal{L}).
\end{equation}
Here, $\mathcal{D}(\mathcal{L})$ denotes the domain of the operator $\mathcal{L}$. Analogous to~\eqref{mcolb_eqn:disc_decomp}, the relations in \eqref{mcolb_eqn:cts_decomp} show how $\mathcal{E}$ decomposes $\mathcal{H}$ and diagonalizes the operator $\mathcal{L}$.

Of particular interest are the (scalar-valued) spectral measures of $\mathcal{L}$ with respect to $f\in\mathcal{H}$, given by $\mu_f(\Omega):=\langle\mathcal{E}(\Omega)f,f\rangle$, for Borel-measurable sets $\Omega\subset\mathbb{R}$. Lebesgue's decomposition of $\mu_f$ is
$$\setlength\abovedisplayskip{6pt}\setlength\belowdisplayskip{6pt}
d\mu_f(y)= \underbrace{\sum_{\lambda\in\Lambda^{{\rm p}}(\mathcal{L})}\langle\mathcal{P}_\lambda f,f\rangle\,\delta({y-\lambda})dy}_{\text{discrete part}}+\underbrace{\rho_f(y)\,dy +d\mu_f^{(\mathrm{sc})}(y)}_{\text{continuous part}}.
$$
The discrete part of $\mu_f$ is a sum of Dirac delta distributions, supported on the set of eigenvalues of $\mathcal{L}$, which we denote by $\Lambda^{{\rm p}}(\mathcal{L})$. The coefficient of each $\delta$ in the sum is $\langle\mathcal{P}_\lambda f,f\rangle=\|\mathcal{P}_\lambda f\|^2$, where $\mathcal{P}_\lambda$ is the orthogonal spectral projector associated with the eigenvalue $\lambda$, and $\|\cdot\|=\sqrt{\langle\cdot,\cdot\rangle}$ is the norm on $\mathcal{H}$. The continuous part of $\mu_f$ consists of an absolutely continuous\footnote{We take ``absolutely continuous'' to be with respect to the Lebesgue measure.} part with Radon--Nikodym derivative $\rho_f\in L^1(\mathbb{R})$ and a singular continuous component $\smash{\mu_f^{(\mathrm{sc})}}$.  Without loss of generality, we assume throughout that $\|f\|=1$, which ensures that $\mu_f$ is a probability measure.

Computing $\mu_f$ is important in many applications, and can be considered an infinite-dimensional analogue of computing eigenvectors. We aim to evaluate smoothed approximations of $\mu_f$. We compute a smooth function $\mu_{f}^\epsilon$, with smoothing parameter $\epsilon>0$, that converges weakly to $\mu_f$~\cite[Ch.~1]{billingsley2013convergence}. That is, 
\begin{equation}\label{mcolb_eqn:weak_conv}\setlength\abovedisplayskip{6pt}\setlength\belowdisplayskip{6pt}
\int_\mathbb{R} \phi(y)\mu_{f}^\epsilon(y)\,dy\rightarrow \int_\mathbb{R}\phi(y)\,d\mu_f(y), \qquad\text{as}\qquad\epsilon\downarrow 0,
\end{equation}
for any bounded, continuous function $\phi$.

\section{Algorithmic framework for \texttt{SpecSolve}}

Our key ingredient is the resolvent $(\mathcal{L}-z)^{-1}=\int_{\Lambda(\mathcal{L})}(\lambda-z)^{-1}d\mathcal{E}(\lambda)$ for $z\not\in\Lambda(\mathcal{L}).$ Stone's formula \cite{stone1932linear} links the resolvent to convolution with the Poisson kernel:
\begin{equation}\label{mcolb_jbjibfvbq}\setlength\abovedisplayskip{6pt}\setlength\belowdisplayskip{6pt}
\mu_{f}^\epsilon(x)=\frac{-1}{\pi}{\rm Im}\left(\langle (\mathcal{L}-(x-\epsilon i))^{-1}f,f \rangle\right)=\int_{\mathbb{R}}{}{\frac{\epsilon\pi^{-1}}{(x-\lambda)^2+\epsilon^2}}\,{d\mu_f(\lambda)}.
\end{equation}
As $\epsilon\downarrow 0$, this approximation converges weakly to $\mu_f$. To compute $(\mathcal{L}-(x-\epsilon i))^{-1}f$ we must somehow discretize the operator. However, for a given discretization size, if $\epsilon$ is too small, the approximation via \eqref{mcolb_jbjibfvbq} becomes unstable \cite[Section 4.3]{colbrook2020} due to the discrete spectrum of the discretization. We must adaptively increase the discretization/truncation size as $\epsilon\downarrow 0$ and there is an increased computational cost for smaller $\epsilon$. Therefore, replacing the Poisson kernel with higher-order rational kernels is advantageous. These kernels have better convergence rates as $\epsilon\downarrow 0$, allowing a larger $\epsilon$ to be used for a given accuracy, and thus a lower computational burden.

Let $\{a_j\}_{j=1}^m$ be distinct points in the upper half plane and suppose that the constants $\{\alpha_j\}_{j=1}^m$ satisfy the following (transposed) Vandermonde system:
\begin{equation}\label{mcolb_eqn:vandermonde_condition}\setlength\abovedisplayskip{6pt}\setlength\belowdisplayskip{6pt}
\begin{pmatrix}
1 & \dots & 1 \\
a_1 & \dots & a_{m} \\
\vdots & \ddots & \vdots \\
a_1^{m-1} &  \dots & a_{m}^{m-1}
\end{pmatrix}
\begin{pmatrix}
\alpha_1 \\ \alpha_2\\ \vdots \\ \alpha_m
\end{pmatrix}
=\begin{pmatrix}
1 \\ 0 \\ \vdots \\0
\end{pmatrix}.
\end{equation}
Then the kernel
$$\setlength\abovedisplayskip{6pt}\setlength\belowdisplayskip{6pt}
{K(x)=\frac{1}{2\pi i}\sum_{j=1}^m\frac{\alpha_j}{x-a_j}-\frac{1}{2\pi i}\sum_{j=1}^m\frac{\overline{\alpha_j}}{x-\overline{a_j}} \quad \text{with }K_\epsilon(x)=\epsilon^{-1}K(x\epsilon^{-1})}
$$
is an $m$th order kernel, and we have the following generalization of Stone's formula
\begin{equation}\setlength\abovedisplayskip{6pt}\setlength\belowdisplayskip{6pt}
\label{mcolb_eqn:egn_stone}
\mu_{f}^\epsilon(x)=[K_{\epsilon}*\mu_f](x)=\frac{-1}{\pi}\sum_{j=1}^{m}{\rm Im}\left(\alpha_j\,\langle (\mathcal{L}-(x-\epsilon a_j))^{-1}f,f \rangle\right).
\end{equation}
This provides $\mathcal{O}(\epsilon^m\log(\epsilon^{-1}))$ convergence in \eqref{mcolb_eqn:weak_conv} if $\phi$ is sufficiently regular, and similar rates for $\mu_{f}^\epsilon(x)\rightarrow\rho_f(x)$ if $\mu_f$ is sufficiently regular near $x$ \cite{colbrook2020}.

\begin{algorithm}[t]
\textbf{Input:} $\mathcal{L}$, $f\in\mathcal{H}$, $x_0\in\mathbb{R}$, $a_1,\dots,a_m\in\{z\in\mathbb{C}:{\rm Im}(z)>0\}$, and $\epsilon>0$. \\
\vspace{-4mm}
\begin{algorithmic}[1]
\STATE Solve the Vandermonde system \eqref{mcolb_eqn:vandermonde_condition} for the residues $\alpha_1,\dots,\alpha_m\in\mathbb{C}$.
	\STATE Solve $(\mathcal{L}-(x_0-\epsilon a_j))u_{j}^\epsilon=f$ for $1\leq j\leq m$.
	\STATE Compute $\smash{\mu_{f}^\epsilon(x_0)=\frac{-1}{\pi}\mathrm{Im}\left(\sum_{j=1}^m\alpha_j\langle u_{j}^\epsilon,f\rangle\right)}$.
\end{algorithmic} \textbf{Output:} The approximate spectral measure $\mu_{f}^\epsilon(x_0)$.
\caption{A computational framework for evaluating an approximate spectral measure of an operator $\mathcal{L}$ at $x_0\in\mathbb{R}$ with respect to a vector $f\in\mathcal{H}$.}\label{mcolb_alg:spec_meas}
\end{algorithm}

We consider the choice $a_j={2j}/({m+1})-1+i$ and the framework for evaluating $\mu_{f}^\epsilon$ is summarized in \Cref{mcolb_alg:spec_meas}. This algorithm forms the foundation of \texttt{SpecSolve} \cite{SpecSolve_code} and can be performed in parallel for several $x_0$. We compute an accurate value of $\mu_{f}^\epsilon$ provided that the resolvent is applied with sufficient accuracy. For an efficient adaptive implementation, \texttt{SpecSolve} constructs a fixed discretization, solves linear systems at each required complex shift, and checks the approximation error at each shift. If further accuracy is needed at a subset of the shifts, then the discretization size is doubled, applied at these shifts, and the error is recomputed. This process is repeated until the resolvent is computed accurately at all shifts.

\section{Singular integro-differential operators}

Singular integral operators of Cauchy type play a pivotal role in the classical theory of PDEs and their spectral properties~\cite{Muskhelishvilli2008}. They appear in a wide range of physical models, along with their integro-differential and nonlinear counterparts~\cite{cuminato2007}. 

Consider the self-adjoint singular integral operator $\mathcal{L}$ in~\eqref{mcolb_eqn:SIO} with $G(x,y)=\overline{G(y,x)}$, and $a(x)$ real, continuously differentiable, and bounded. To compute spectral measures of $\mathcal{L}$ in the \texttt{SpecSolve} framework, we must compute inner products between functions in $L^2(\mathbb{R})$ and solve linear equations with a complex shift $z$, e.g.,
\begin{equation}\label{mcolb_eqn:shift_SIE}\setlength\abovedisplayskip{6pt}\setlength\belowdisplayskip{6pt}
(a(x)-z)u(x) + \frac{1}{\pi i}\int_\mathbb{R}\frac{G(x,y)}{y-x}\,u(y)\,dy = f(x).
\end{equation}
We discretize $L^2(\mathbb{R})$ with the orthogonal rational basis functions $\rho_n(x) = \frac{1}{\sqrt{\pi}}\frac{(1+ix)^n}{(1-ix)^{n+1}}$, for $n\in\mathbb{Z}$. These functions have excellent approximation properties, are associated with banded differentiation and multiplication matrices, and expansion coefficients can be computed from function samples in quasi-linear time with the FFT~\cite{Iserles2020}. Moreover, they diagonalize the Hilbert transform~\cite{Weideman1995} and lead to banded discretizations of~\eqref{mcolb_eqn:shift_SIE} when $G(x,y)$ is sufficiently smooth and of low numerical rank~\cite{Slevinsky2017}.

Both the multiplicative and integral components of $\mathcal{L}$ can contribute continuous spectrum. When $G(x,y)=k(x)k(y)$ is a rank one kernel with $k(x)>0$, the spectrum fills the interval $[\min|a(x)-k(x)|,\max|a(x)+k(x)|]$~\cite{koppelman1960}. \Cref{mcolb_fig:SIOs} (left) shows the spectral measures $\mu_f$ of $\mathcal{L}$, with $\smash{f(x)=\sqrt{2/\pi}(1-x^2)^{-1}}$, $\smash{k(x)=e^{-x^2}}$, and $a_\pm(x)= \pm 2/(1+x^2)^2$. The dashed grey lines highlight the support of the measures in the expected interval. We can also tackle singular integro-differential operators. \Cref{mcolb_fig:SIOs} (right) compares the spectral measures of $-d^2/dx^2$ and $-d^2/dx^2 + (1/\pi i)\int_\mathbb{R}(y-x)^{-1}dy$ with respect to $f$. Both the second derivative and the singular integral are diagonalized by the Fourier transform, and the spectral measures can be computed analytically (dashed lines). The integral perturbation breaks the symmetry between positive and negative Fourier modes, which effectively splits the spectral measure of $-d^2/dx^2$ into two duplicate peaks of half height at $\pm 1$.

\begin{figure}[!tbp]
 \centering
 \begin{minipage}[b]{0.49\textwidth}
  \begin{overpic}[width=\textwidth,trim={0mm 0mm 0mm 0mm},clip]{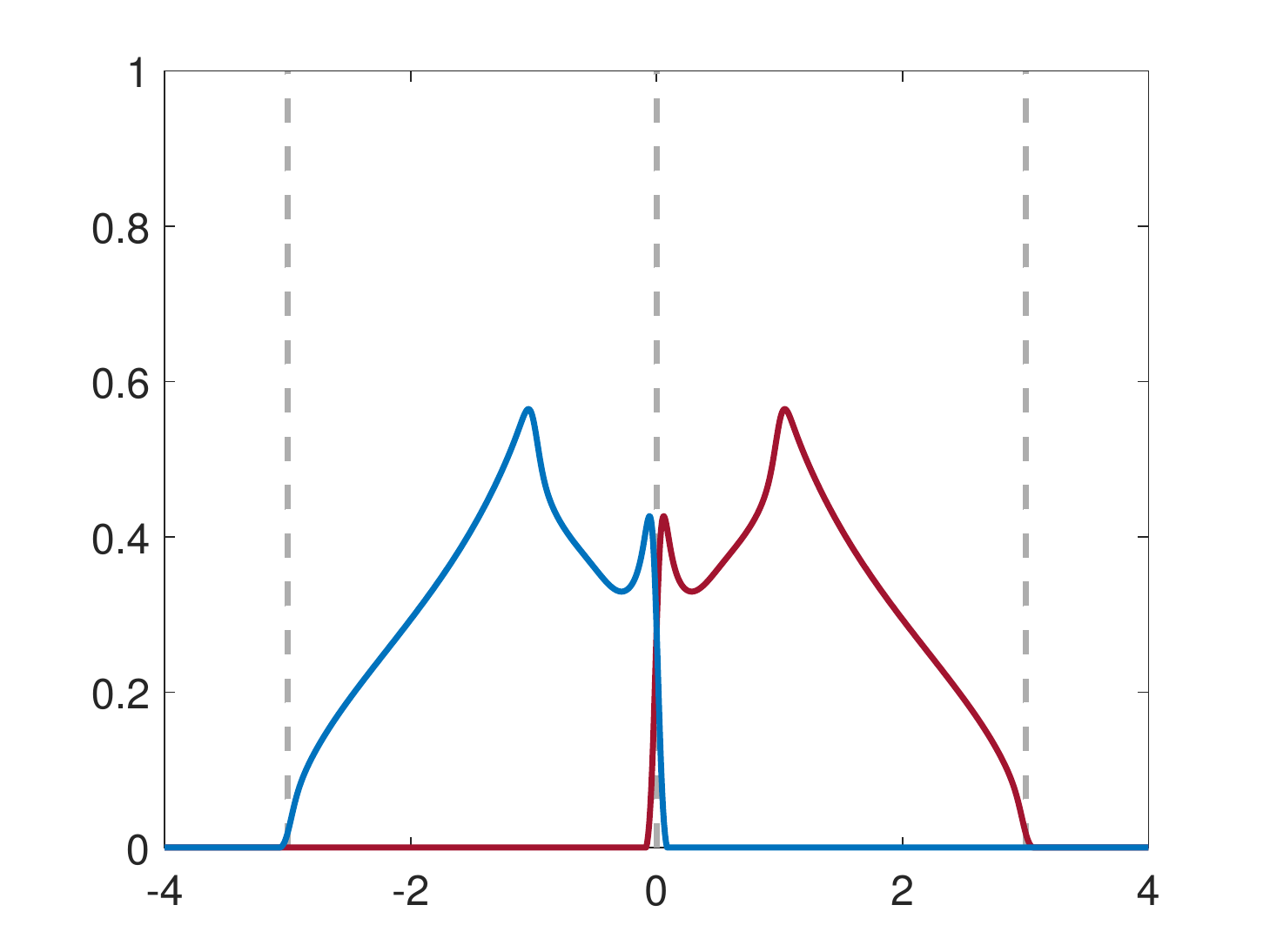}
	\put (45,73) {$\mu_{f}^{0.1}(x)$}
   		\put (50,-1) {$x$}
		\put(27,30){\rotatebox{54}{$a_-(x)$}}
		\put(64,40){\rotatebox{-52}{$a_+(x)$}}
   \end{overpic}
 \end{minipage}
	\begin{minipage}[b]{0.49\textwidth}
  \begin{overpic}[width=\textwidth,trim={0mm 0mm 0mm 0mm},clip]{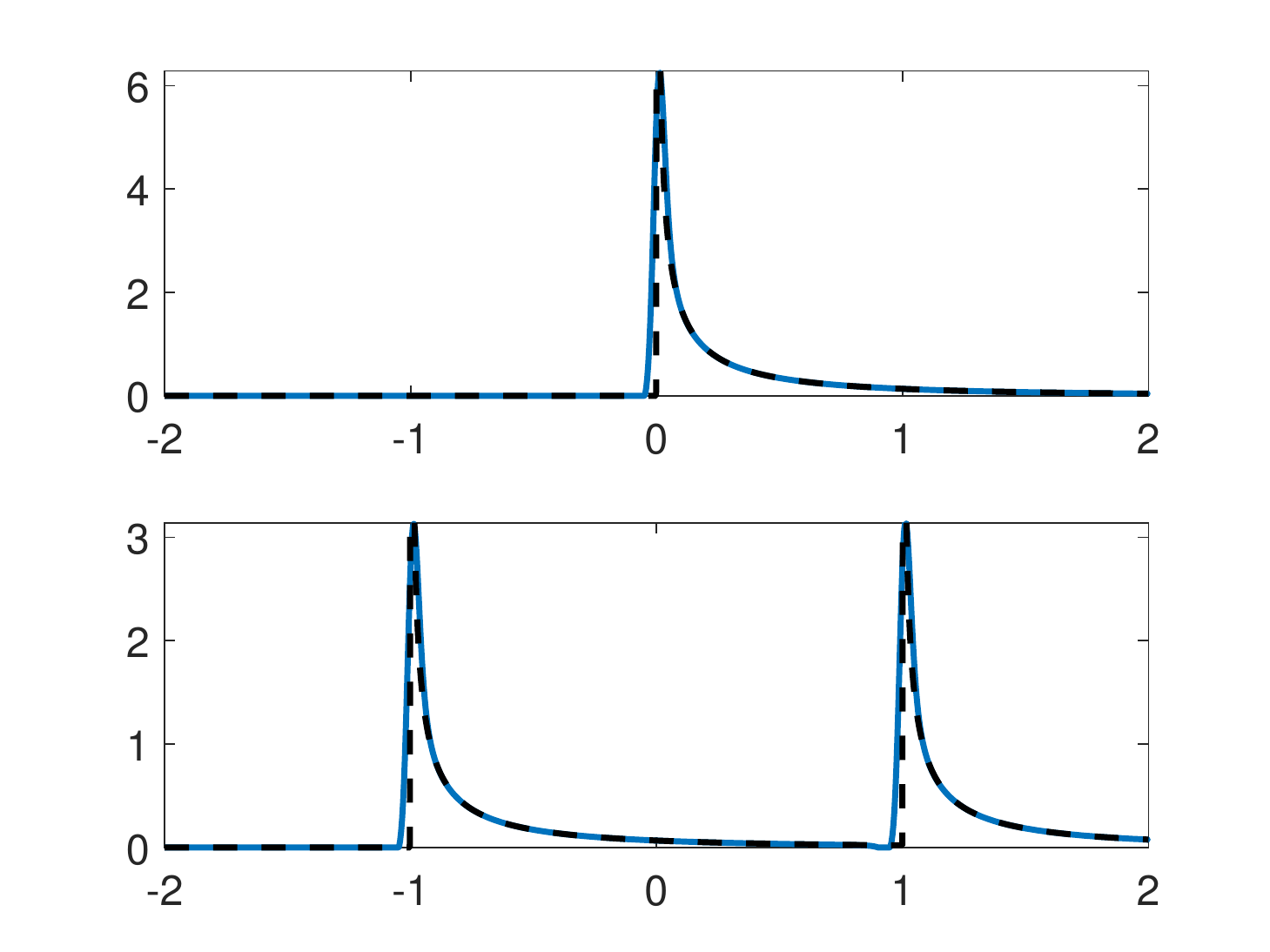}
  		\put (45,73) {$\mu_{f}^{0.05}(x)$}
   		\put (50,-1) {$x$}
   \end{overpic}
 \end{minipage}
	  \caption{Left: The smoothed spectral measures, $\smash{\mu_f^{0.1}}$, computed with a $4$th order kernel are supported on the intervals $\smash{[\min|a(x)-k(x)|,\max|a(x)+k(x)|]}$. Right: The smoothed spectral measures, $\smash{\mu_f^{0.05}}$, computed with a $4$th order kernel for $-d^2/dx^2$ (top) and $\smash{-d^2/dx^2+(1/\pi i)\int_\mathbb{R}(y-x)^{-1}dy}$ (bottom) are compared with analytical solutions (dashed lines).}
\label{mcolb_fig:SIOs}
\end{figure}

\begin{figure}[!tbp]
 \centering
 \begin{minipage}[b]{0.32\textwidth}
  \begin{overpic}[width=\textwidth,trim={0mm 0mm 0mm 0mm},clip]{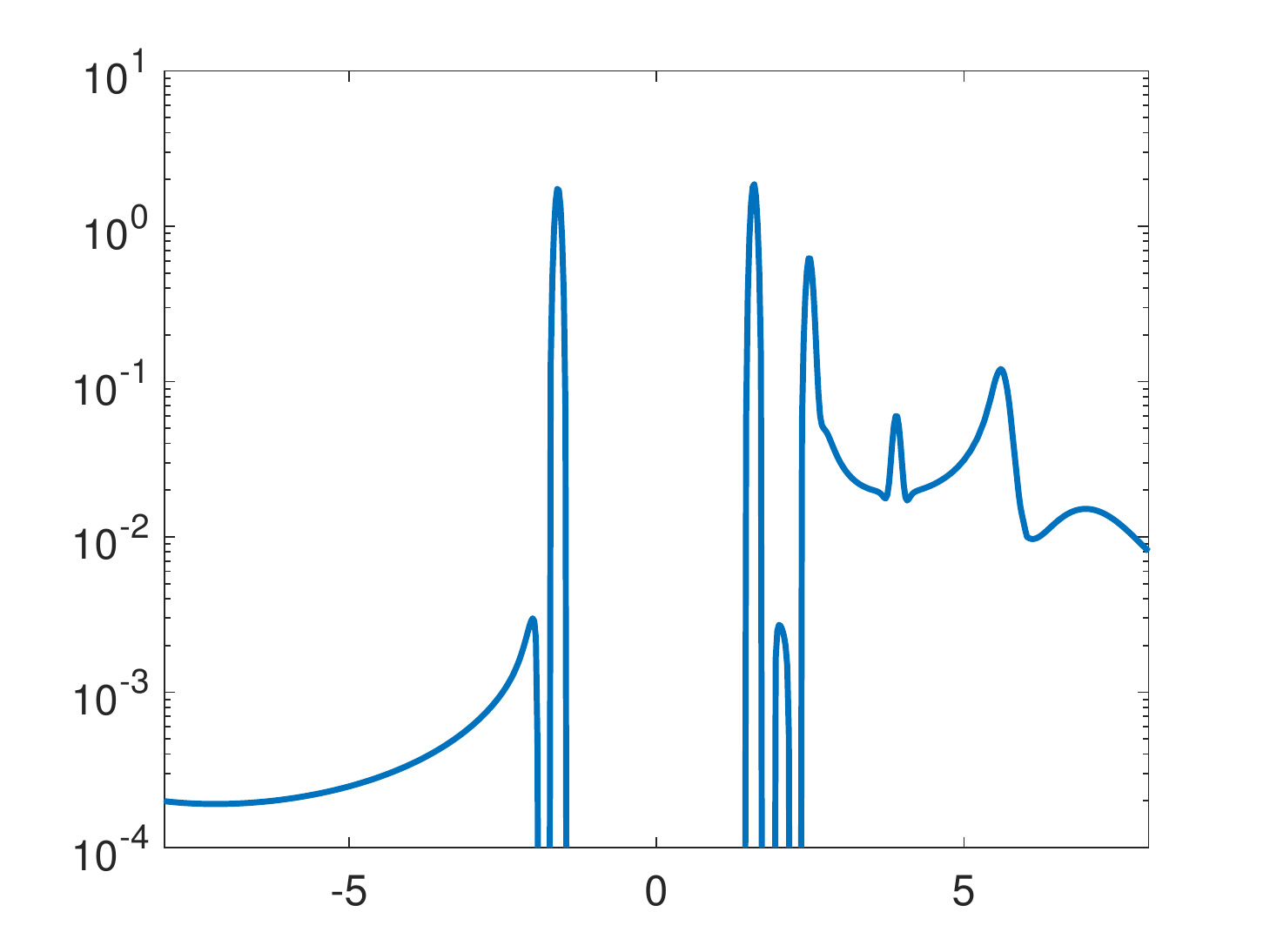}
	\put (45,74) {$\mu_{f}^{0.2}(x)$}
   		\put (50,-1) {$x$}
   \end{overpic}
 \end{minipage}
	\begin{minipage}[b]{0.32\textwidth}
  \begin{overpic}[width=\textwidth,trim={0mm 0mm 0mm 0mm},clip]{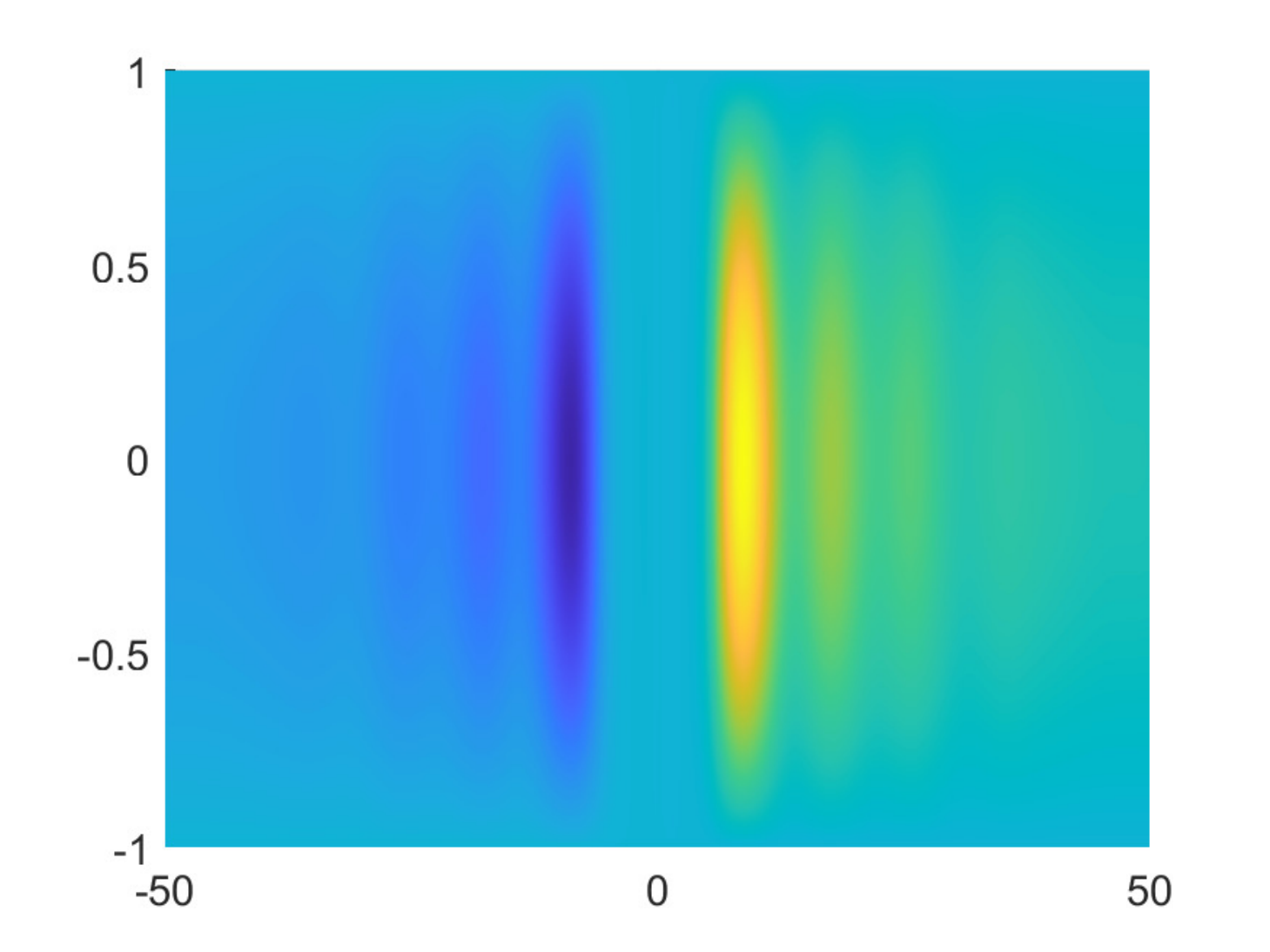}
  		\put (15,74) {$[\mathcal{E}([2.1,3]){f}](x,y)$}
   		\put (50,-1) {$x$}
   		\put (2,38) {$y$}
   \end{overpic}
 \end{minipage}
\begin{minipage}[b]{0.32\textwidth}
  \begin{overpic}[width=\textwidth,trim={0mm 0mm 0mm 0mm},clip]{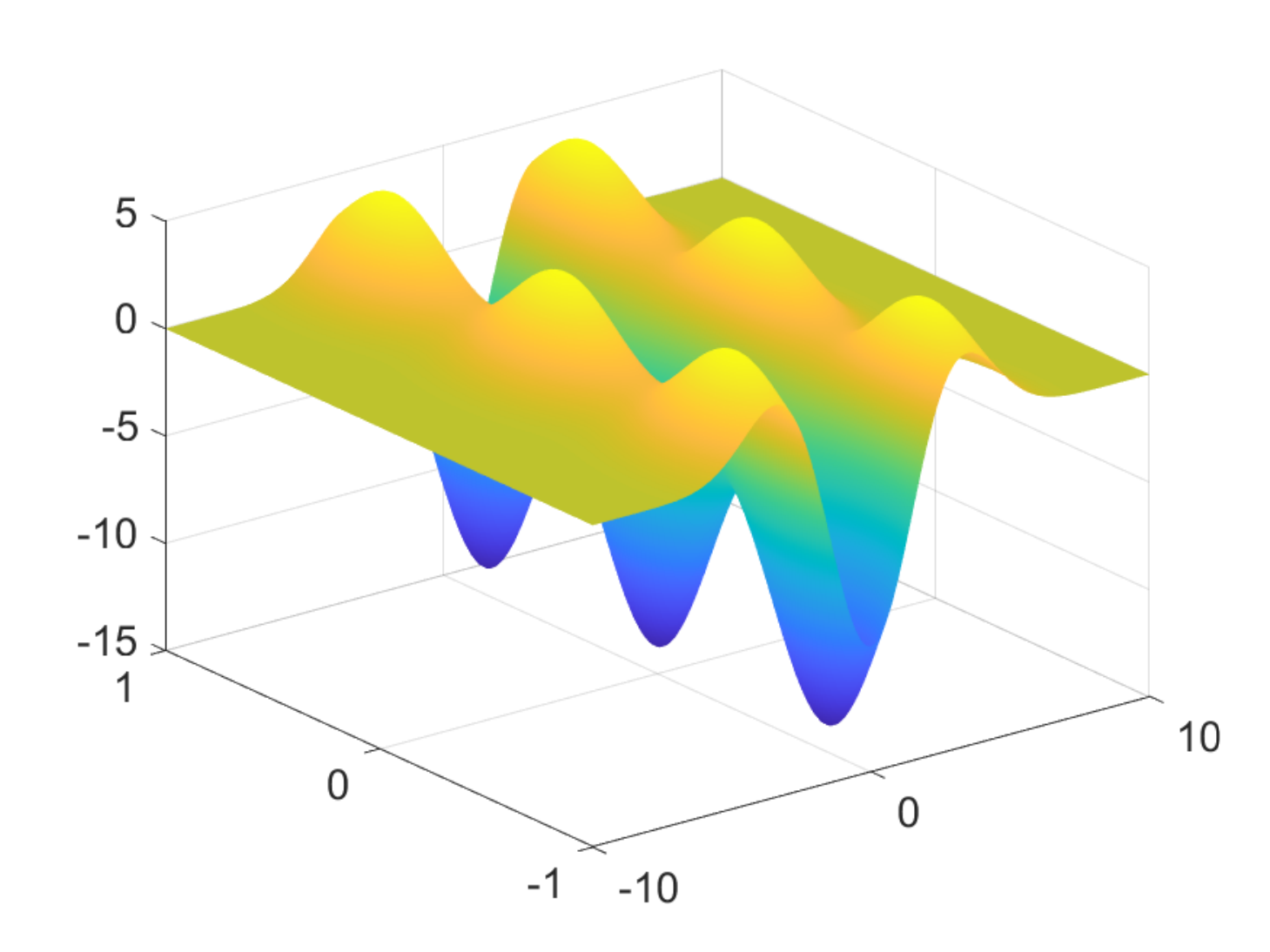}
	\put (20,74) {$v(x,y)$}
   	\put (68,4) {$x$}
  	\put (20,6) {$y$}
   \end{overpic}
 \end{minipage}
	  \caption{Left: The smoothed spectral measure, $\smash{\mu_f^{0.2}}$, of the partial integro-differential operator in~\eqref{mcolb_eqn:PIDO} computed with a $6$th order kernel. Middle: Spectral projection $\mathcal{E}([2.1,3])f$ of $f(x,y) =(1+x)(1+x^2)^{-1}\cos(\pi y/2)/\sqrt{\pi}$ associated with the third resonance peak from the left in the plot of $\mu_f^{0.2}$. Right: The potential energy landscape $v(x,y)$ for the operator in~\eqref{mcolb_eqn:PIDO}.}
\label{mcolb_fig:PIDO}
\end{figure}
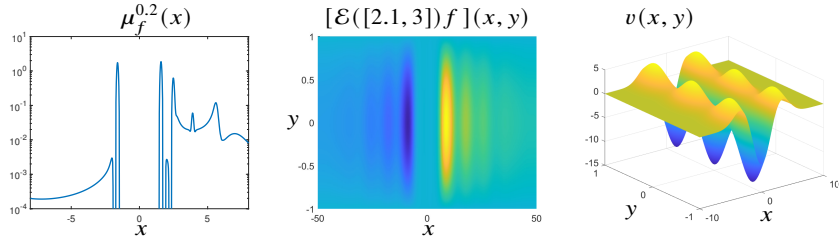

The \texttt{SpecSolve} framework can also compute spectral projections $\mathcal{E}([a,b])$ associated with the projection-valued measure by omitting the inner product step in~\Cref{mcolb_alg:spec_meas} and applying endpoint corrections~\cite{Colbrook2021}. \Cref{mcolb_fig:PIDO} displays a scalar spectral measure and spectral projection for the partial integro-differential operator
\begin{equation}\label{mcolb_eqn:PIDO}\setlength\abovedisplayskip{6pt}\setlength\belowdisplayskip{6pt}
-\Delta u + v(x,y) u + \frac{1}{\pi i}\int_\mathbb{R}\frac{\exp(-x^2-y^2)}{y-x}\,u(\cdot,y)\,dy,\quad \mathcal{H}=L^2(\mathbb{R}\times[-1, 1]),
\end{equation}
and the function $f(x,y) =(1+x)(1+x^2)^{-1}\cos(\pi y/2)/\sqrt{\pi}$. The potential function is $v(x,y)$ is also plotted in \cref{mcolb_fig:PIDO}. The operator is discretized with a tensor product basis of the rational orthogonal functions $\{\rho_n\}$ and ultraspherical polynomials~\cite{Olver2013}, resulting in a sparse and banded discretization (we use basis reordering to reduce the bandwidth). In~\cref{mcolb_fig:PIDO}, narrow peaks in the scalar measure reveal scattering resonances of the partial integro-differential operator and the associated spectral projections uncover wave-packet modes that are highly concentrated within the potential well.

\section{Linear operator pencils}

For matrices $A,B\in\mathbb{C}^n$, the classical generalized eigenvalue problem is the problem of finding $v\in\mathbb{C}^n\backslash\{0\}$ and $\lambda\in\mathbb{C}$ such that $Av=\lambda Bv.$ For example, this problem arises in finite element discretizations of eigenproblems for elliptic partial differential operators, where $A$ corresponds to the ``stiffness matrix'' and $B$ corresponds to the ``mass matrix'' \cite{boffi2010finite}. Another example is linearization methods for non-linear eigenvalue problems \cite{guttel2017nonlinear}. For many applications, $A$ and $B$ are finite approximations of (possibly unbounded) operators $\mathcal{A}$ and $\mathcal{B}$ acting on a separable Hilbert space. We consider the case that $\mathcal{A}$ and $\mathcal{B}$ are both self-adjoint and that $\mathcal{B}$ is positive and invertible. We study the generalized spectral problem through the operator formally defined as $\mathcal{L}=\mathcal{B}^{-1}\mathcal{A}.$

\subsection{Recovering a self-adjoint operator}
\label{mcolb_technical_section}

It is well-known that $\mathcal{D}(\mathcal{B}^{1/2})$ is complete with respect to the norm $\|f\|_{\mathcal{B}}:=\langle \mathcal{B}^{1/2}f,\mathcal{B}^{1/2}f\rangle$ \cite[Theorem 4.4.2]{davies1996spectral}. We denote the induced Hilbert space by $\mathcal{H}_{\mathcal{B}}$. The operator $\mathcal{B}^{-1}\mathcal{A}$ with domain $\mathcal{D}(\mathcal{A})\cap\mathcal{D}(\mathcal{B}^{1/2})$ is symmetric in $\mathcal{H}_{\mathcal{B}}$.\footnote{Suppose that $f,g\in\mathcal{D}(\mathcal{A})\cap\mathcal{D}(\mathcal{B}^{1/2})$. Then $\langle\mathcal{B}^{1/2}(\mathcal{B}^{-1}\mathcal{A})g,\mathcal{B}^{1/2}f\rangle=\langle\mathcal{B}^{-1/2}\mathcal{A}g,\mathcal{B}^{1/2}f\rangle=\langle \mathcal{A}g,f\rangle. $ The first equality follows since $\mathcal{B}^{-1}\mathcal{A}g\in\mathcal{D}(\mathcal{B}^{1/2})$, whereas the second follows since $\mathcal{B}^{-1/2}$ is a bounded self-adjoint operator on $\mathcal{H}$. Similarly, we have that $\langle\mathcal{B}^{1/2}g,\mathcal{B}^{1/2}(\mathcal{B}^{-1}\mathcal{A})f\rangle=\langle g,\mathcal{A}f\rangle$.} However, to apply the spectral theorem, we need a self-adjoint operator. We assume that $\mathcal{D}(\mathcal{A})\cap\mathcal{D}(\mathcal{B}^{1/2})$ is a dense subspace of the Hilbert space $\mathcal{H}_{\mathcal{B}}$. Since $\mathcal{B}^{-1}\mathcal{A}$ is symmetric in $\mathcal{H}_{\mathcal{B}}$, it is closable. We define the symmetric closed operator
\begin{equation}\setlength\abovedisplayskip{6pt}\setlength\belowdisplayskip{6pt}
\label{mcolb_operL}
\mathcal{L}=\overline{\mathcal{B}^{-1}\mathcal{A}|_{\mathcal{D}(\mathcal{A})\cap\mathcal{D}(\mathcal{B}^{1/2})}},
\end{equation}
where the closure is performed with respect to $\mathcal{H}_{\mathcal{B}}$. This allows us to perform \textit{numerical computations} with $\mathcal{L}$ by restricting to the subspace $\mathcal{D}(\mathcal{A})\cap\mathcal{D}(\mathcal{B}^{1/2})$. To do this, we consider the inner product space $\{f:f\in\mathcal{H}\}$ with inner product $
\langle f,g \rangle_{\mathcal{B}^{-1}}=\langle \mathcal{B}^{-1/2}f,\mathcal{B}^{-1/2}g \rangle.$ We take the completion of this space, $\mathcal{H}_{\mathcal{B}^{-1}}$. $\mathcal{D}(\mathcal{B})$ is dense in $\mathcal{H}_{\mathcal{B}}$ and hence $\mathcal{B}$ can be extended to an invertible isometry from $\mathcal{H}_{\mathcal{B}}$ to $\mathcal{H}_{\mathcal{B}^{-1}}$, and $\mathcal{H}_{\mathcal{B}^{-1}}$ can be identified with the dual of $\mathcal{H}_{\mathcal{B}}$. We assume that $\mathcal{A}|_{\mathcal{D}(\mathcal{A})\cap\mathcal{D}(\mathcal{B}^{1/2})}:\mathcal{H}_{\mathcal{B}}\rightarrow \mathcal{H}_{\mathcal{B}^{-1}}$ is closable, with closure denoted by $\mathcal{A}_{\mathcal{B}}$. We can now define
\begin{equation}\setlength\abovedisplayskip{6pt}\setlength\belowdisplayskip{6pt}
\begin{split}
&\mathcal{T}(z):\mathcal{D}(\mathcal{A}_{\mathcal{B}})\rightarrow \mathcal{H}_{\mathcal{B}^{-1}}, \quad f\hookrightarrow (\mathcal{A}_{\mathcal{B}}-z\mathcal{B})f,\\
&\Lambda(\mathcal{A},\mathcal{B})=\{z\in\mathbb{C}:{\mathcal{T}(z)}\text{ does not have bounded inverse}\}.
\end{split}
\end{equation}

\begin{proposition}
\label{mcolb_needed_prop}
For any $z\in\mathbb{C}$, $\mathcal{D}(\mathcal{T}(z))=\mathcal{D}(\mathcal{L})$ and $\mathcal{T}(z)=\mathcal{B}(\mathcal{L}-z).$ Moreover, $\Lambda(\mathcal{A},\mathcal{B})=\Lambda(\mathcal{L})$ and if $z\in\mathbb{C}\setminus\Lambda(\mathcal{L})$, then $(\mathcal{L}-z)^{-1}=\mathcal{T}(z)^{-1}\mathcal{B}.$
\end{proposition}

\begin{proof}
Let $z\in\mathbb{C}$ and $f\in\mathcal{D}(\mathcal{L})$. Then there exists $f_n\in \mathcal{D}(\mathcal{A})\cap\mathcal{D}(\mathcal{B}^{1/2})$ such that $\lim_{n\rightarrow\infty}f_n=f$ (in $\mathcal{H}_{\mathcal{B}}$) and $\lim_{n\rightarrow\infty}(\mathcal{L}-z)f_n=(\mathcal{L}-z)f$ (in $\mathcal{H}_{\mathcal{B}}$). Since $\|\mathcal{T}(z)f_n-\mathcal{T}(z)f_m\|_{\mathcal{B}^{-1}}=\|(\mathcal{L}-z)f_n-(\mathcal{L}-z)f_m\|_{\mathcal{B}},$ it follows that $\{\mathcal{T}(z)f_n\}$ is Cauchy in $\mathcal{H}_{\mathcal{B}^{-1}}$ and hence converges to some $g\in\mathcal{H}_{\mathcal{B}^{-1}}$. Since $\mathcal{T}(z)$ is closed, $f\in\mathcal{D}(\mathcal{T}(z))$ and $\mathcal{T}(z)f=g.$ Moreover, $\|(\mathcal{L}-z)f_n-\mathcal{B}^{-1}g\|_{\mathcal{B}}=\|\mathcal{T}(z)f_n-g\|_{\mathcal{B}^{-1}}$ converges to zero. Since $\mathcal{L}$ is closed, $(\mathcal{L}-z)f=\mathcal{B}^{-1}g$ and hence that $\mathcal{B}(\mathcal{L}-z)f=\mathcal{T}(z)f$. A similar argument shows that $\mathcal{D}(\mathcal{T}(z))\subset\mathcal{D}(\mathcal{L})$. Hence, $\mathcal{D}(\mathcal{T}(z))=\mathcal{D}(\mathcal{L})$ and $\mathcal{T}(z)=\mathcal{B}(\mathcal{L}-z).$ The proposition follows since $\mathcal{B}:\mathcal{H}_{\mathcal{B}}\rightarrow \mathcal{H}_{\mathcal{B}^{-1}}$ is an isometry.
\end{proof}

The following theorem that gives sufficient conditions for $\mathcal{L}$ to be self-adjoint. Common examples of these conditions include when $\mathcal{A}$ and $\mathcal{B}$ are suitable elliptic PDEs of the same differentiation order (see condition (C1)), $\mathcal{A}$ is bounded (see condition (C2)), and $\mathcal{B}$ is a suitable weight function (see condition (C3))

\begin{theorem}
\label{mcolb_some_conditions}
Consider the operators $\mathcal{A}$, $\mathcal{B}$ and $\mathcal{L}$ above. Suppose that any of the following conditions hold:
\begin{enumerate}
	\item[(C1)] There exist constants $a,b>0$ such that for any $f\in\mathcal{D}(\mathcal{A})\cap\mathcal{D}(\mathcal{B}^{1/2})$
	\begin{equation}\setlength\abovedisplayskip{6pt}\setlength\belowdisplayskip{6pt}
	\label{mcolb_bound1}
	\|\mathcal{B}^{-1/2}\mathcal{A}f\|\leq a\|f\|+b\|\mathcal{B}^{1/2}f\|.
	\end{equation}
	\item[(C2)] $\mathcal{A}$ is a relatively bounded perturbation of $\mathcal{B}$, meaning that $\mathcal{D}(\mathcal{B})\subset\mathcal{D}(\mathcal{A})$ and there exist constants $a,b>0$ such that for any $f\in\mathcal{D}(\mathcal{B})$
	\begin{equation}\setlength\abovedisplayskip{6pt}\setlength\belowdisplayskip{6pt}
	\label{mcolb_bound2}
	\|\mathcal{A}f\|\leq a\|f\|+b\|\mathcal{B}f\|.
	\end{equation}
	\item[(C3)] $\mathrm{Sp}(\mathcal{A})\neq\mathbb{R}$ and $\mathcal{B}$ is a relatively bounded perturbation of $\mathcal{A}$, meaning that $\mathcal{D}(\mathcal{A})\subset\mathcal{D}(\mathcal{B})$ and there exist constants $a,b>0$ such that for any $f\in\mathcal{D}(\mathcal{A})$
	\begin{equation}\setlength\abovedisplayskip{6pt}\setlength\belowdisplayskip{6pt}
	\label{mcolb_bound2b}
	\|\mathcal{B}f\|\leq a\|f\|+b\|\mathcal{A}f\|.
	\end{equation}
\end{enumerate}
Then $\mathcal{L}$ is self-adjoint on $\mathcal{H}_{\mathcal{B}}$. Moreover, when $\mathrm{(C1)}$ holds, $\mathcal{L}$ is bounded.
\end{theorem}

\begin{proof}
Suppose first that $(C1)$ holds. Since $\mathcal{B}^{1/2}$ is strictly positive, \eqref{mcolb_bound1} implies that there exists a positive constant $c$ such that $
\|\mathcal{B}^{-1/2}\mathcal{A}f\|\leq c\|\mathcal{B}^{1/2}f\|$ for any $f\in\mathcal{D}(\mathcal{A})\cap\mathcal{D}(\mathcal{B}^{1/2})$. This is equivalent to boundedness of $\mathcal{B}^{-1}\mathcal{A}|_{\mathcal{D}(\mathcal{A})\cap\mathcal{D}(\mathcal{B}^{1/2})}$ in the Hilbert space $\mathcal{H}_{\mathcal{B}}$, and hence $\mathcal{L}$ is bounded and self-adjoint on $\mathcal{H}_{\mathcal{B}}$.

For $(C2)$ or $(C2)$, we claim that it is enough to show that there exists some $\gamma>0$ and $\kappa\in\mathbb{R}$ such that the operators
\begin{equation}\setlength\abovedisplayskip{6pt}\setlength\belowdisplayskip{6pt}
\label{mcolb_trick1}
T_{\pm}=\mathcal{A}+\kappa I\pm i\gamma\mathcal{B},\quad \mathcal{D}(T_{\pm})=\mathcal{D}(\mathcal{A})\cap\mathcal{D}(\mathcal{B})
\end{equation}
are closable (in $\mathcal{H}$), and that their closures, denoted $\mathcal{T}_{\pm}$, are invertible (in $\mathcal{H}$). To see this, suppose that these conditions hold. Let $g\in\mathcal{D}(\mathcal{B})$ and set $f^{\pm}=T_{\pm}^{-1}\mathcal{B}g.$ Then, by definition of the closure, there exists $f_n^{\pm}\in\mathcal{D}(\mathcal{A})\cap\mathcal{D}(\mathcal{B})\subset\mathcal{D}(\mathcal{A})\cap\mathcal{D}(\mathcal{B}^{1/2})$ such that $f_n^{\pm}\rightarrow f^{\pm}$ and $T_{\pm}f_n^{\pm}\rightarrow \mathcal{B}g$ as $n\rightarrow\infty$ (with convergence in $\mathcal{H}$). Thus,
$\|\mathcal{B}^{-1}(\mathcal{A}+\kappa I)f^{\pm}_n\pm i\gamma f^{\pm}_n-g\|_{\mathcal{B}}=\|\mathcal{B}^{-1/2}\left(T_{\pm}f_n^{\pm}-\mathcal{B}g\right)\|\rightarrow 0$ as $n\rightarrow\infty.$ Since $\mathcal{D}(\mathcal{B})$ is dense in $\mathcal{H}_{\mathcal{B}}$ , it follows that the ranges of $\gamma^{-1}\mathcal{B}^{-1}(\mathcal{A}+ \kappa I)\pm iI$ are also dense in $\mathcal{H}_{\mathcal{B}}$. It follows that $\gamma^{-1}\mathcal{B}^{-1}(\mathcal{A}+\kappa I)$ is essentially self-adjoint in $\mathcal{H}_{\mathcal{B}}$ \cite[p. 257]{reed1972methods}, and hence so is $\mathcal{B}^{-1}\mathcal{A}$. This proves the claim.

Now suppose that $(C2)$ holds. Since $\mathcal{B}$ is strictly positive, \eqref{mcolb_bound2} implies that there exists a positive constant $c<1$ and $\gamma>0$ such that $\|\mathcal{A}f\|\leq c\|\gamma\mathcal{B}f\|$ for any $f\in\mathcal{D}(\mathcal{B})\subset\mathcal{D}(\mathcal{A})$. Hence $\mathcal{A}$ is a relatively bounded perturbation of $i\gamma\mathcal{B}$, with $i\gamma\mathcal{B}$-bound less than $1$. Stability of bounded invertibility \cite[Theorem IV.4.1.16]{kato2013perturbation} implies that $T_{\pm}$ in \eqref{mcolb_trick1} (with $\kappa=0$) are closed and invertible (in $\mathcal{H}$).

Finally, suppose that $(C3)$ holds. Choose $\kappa\in\mathbb{R}$ with $-\kappa\not \in \mathrm{Sp}(\mathcal{A})$ so that $\mathcal{A}+\kappa I$ is invertible, and set $\mathcal{C}=\mathcal{A}+\kappa I$. For any $f\in\mathcal{D}(\mathcal{A})$ and $\gamma>0$, \eqref{mcolb_bound2b} implies that
\begin{equation}\setlength\abovedisplayskip{6pt}\setlength\belowdisplayskip{6pt}
\label{mcolb_need_bound}
	\|\gamma\mathcal{B}f\|\leq \gamma(a+|\kappa|)\|f\|+\gamma b\|\mathcal{C}f\|.
\end{equation}
Choose $\gamma>0$ so that $\gamma(a+|\kappa|)\|\mathcal{C}^{-1}\|+\gamma b<1.$ The stability of bounded invertibility \cite[Theorem IV.4.1.16]{kato2013perturbation} and \eqref{mcolb_need_bound} imply that $T_{\pm}$ are closed and invertible.
\end{proof}

\subsection{Framework for generalized spectral measures}

To extend \texttt{SpecSolve} to the above pencil problem, we simply apply \eqref{mcolb_eqn:egn_stone} with the operator $\mathcal{L}$ defined in \eqref{mcolb_operL} and the Hilbert space $\mathcal{H}_{\mathcal{B}}$. We suppose for simplicity that $f\in\mathcal{D}(\mathcal{B})$. Using \cref{mcolb_needed_prop} and \eqref{mcolb_eqn:egn_stone} and the self-adjointness of $\mathcal{B}^{1/2}$, we have
\begin{equation}\setlength\abovedisplayskip{6pt}\setlength\belowdisplayskip{6pt}
\label{mcolb_eqn:egn_stone2}
\begin{split}
\mu_{f}^\epsilon(x)=[K_{\epsilon}*\mu_f](x)&=\frac{-1}{\pi}\sum_{j=1}^{m}{\rm Im}\left(\alpha_j\,\langle (\mathcal{B}^{1/2}\mathcal{T}(x-\epsilon a_j)^{-1}\mathcal{B}f,\mathcal{B}^{1/2}f \rangle\right)\\
&=\frac{-1}{\pi}\sum_{j=1}^{m}{\rm Im}\left(\alpha_j\,\langle (\mathcal{T}(x-\epsilon a_j)^{-1}\mathcal{B}f,\mathcal{B}f \rangle\right),
\end{split}
\end{equation}
where we use that $\mathcal{B}^{1/2}$ is self-adjoint in the second line and $\langle\cdot,\cdot\rangle$ denotes the inner product on $\mathcal{H}$. This leads to \Cref{mcolb_alg:spec_meas2}, which generalizes \Cref{mcolb_alg:spec_meas}. To apply \Cref{mcolb_alg:spec_meas2}, we only need to compute approximations of $g=\mathcal{B}f$, solve the systems $(\mathcal{A}-(x_0-\epsilon a_j)\mathcal{B})u_{j}^\epsilon=g$, and then compute inner products. We approximate $u_{j}^\epsilon$ using spectral methods and compute inner products using quadrature.

\begin{algorithm}[t]
\textbf{Input:} $\mathcal{A}$, $\mathcal{B}$, $f\in\mathcal{D}(\mathcal{B})$, $x_0\in\mathbb{R}$, $a_1,\dots,a_m\in\{z\in\mathbb{C}:{\rm Im}(z)>0\}$, and $\epsilon>0$. \\
\vspace{-4mm}
\begin{algorithmic}[1]
\STATE Compute $g=\mathcal{B}f$.
\STATE Solve the Vandermonde system \eqref{mcolb_eqn:vandermonde_condition} for the residues $\alpha_1,\dots,\alpha_m\in\mathbb{C}$.
	\STATE Solve $(\mathcal{A}-(x_0-\epsilon a_j)\mathcal{B})u_{j}^\epsilon=g$ for $1\leq j\leq m$.
	\STATE Compute $\smash{\mu_{f}^\epsilon(x_0)=\frac{-1}{\pi}\mathrm{Im}\left(\sum_{j=1}^m\alpha_j\langle u_{j}^\epsilon,g\rangle\right)}$.
\end{algorithmic} \textbf{Output:} The approximate spectral measure $\mu_{f}^\epsilon(x_0)$.
\caption{A computational framework for evaluating an approximate spectral measure of an operator $\mathcal{L}$ in \eqref{mcolb_operL} corresponding to the pencil $\mathcal{A}-\lambda\mathcal{B}$ at $x_0\in\mathbb{R}$ with respect to a vector $f\in\mathcal{D}(\mathcal{B})$.}\label{mcolb_alg:spec_meas2}
\end{algorithm}

\subsection{Examples}

We now present two examples, using Fourier spectral methods and a spectral element method, respectively. Both examples fall into the setup of \Cref{mcolb_some_conditions}.

\vspace{1mm}

\noindent\textbf{Pseudo-differential operators and internal waves:} Spectral properties of 0th order pseudo-differential operators arise naturally in fluid mechanics \cite{Ralston73} and pseudoparabolic equations \cite{showalter1970pseudoparabolic}. See \cite{CS-L20,CdV19} for the study of internal waves and \cite{Zworski1,Zworski2} for connections with scattering resonances. As a simple example, we consider
$$\setlength\abovedisplayskip{6pt}\setlength\belowdisplayskip{6pt}
\mathcal{A} = -i (1+\cos(x)/2)\partial_y, \quad  \mathcal{B} = (1-\partial_y^2)^{1/2},\quad x,y\in[-\pi,\pi]_{\mathrm{per}},
$$
where the initial Hilbert space is $\mathcal{H}=L^2([-\pi,\pi]_{\mathrm{per}}^2)$. To solve the linear systems in \Cref{mcolb_alg:spec_meas2}, we use the standard tensor product Fourier basis.

\Cref{mcolb_fig:fourier} (left) shows the smoothed spectral measures computed using $\epsilon=0.01$, and the first and sixth-order kernels for $f(x,y)=C\exp(\sin(x+y))/(2+\cos(y))$, where $C$ is a normalization constant so that $\mu_f$ is a probability measure. The spectral measure has an absolutely continuous component (with piecewise continuous Radon--Nikodym derivative), and an eigenvalue at $0$. The higher order kernel ($m=6$) provides a better localization of the singular part of the spectral measure at $0$, and also a better resolution of jumps in the Radon--Nikodym derivative (see zoomed-in section). \Cref{mcolb_fig:fourier} (right) shows the pointwise convergence to the Radon--Nikodym derivative and the expected rates of convergence for $m=2,4$ and $6$.

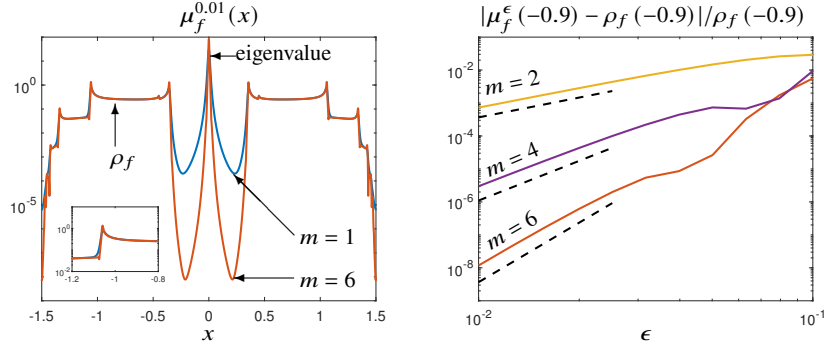
\begin{figure}[!tbp]
 \centering
 \begin{minipage}[b]{0.49\textwidth}
  \begin{overpic}[width=\textwidth,trim={0mm 0mm 0mm 0mm},clip]{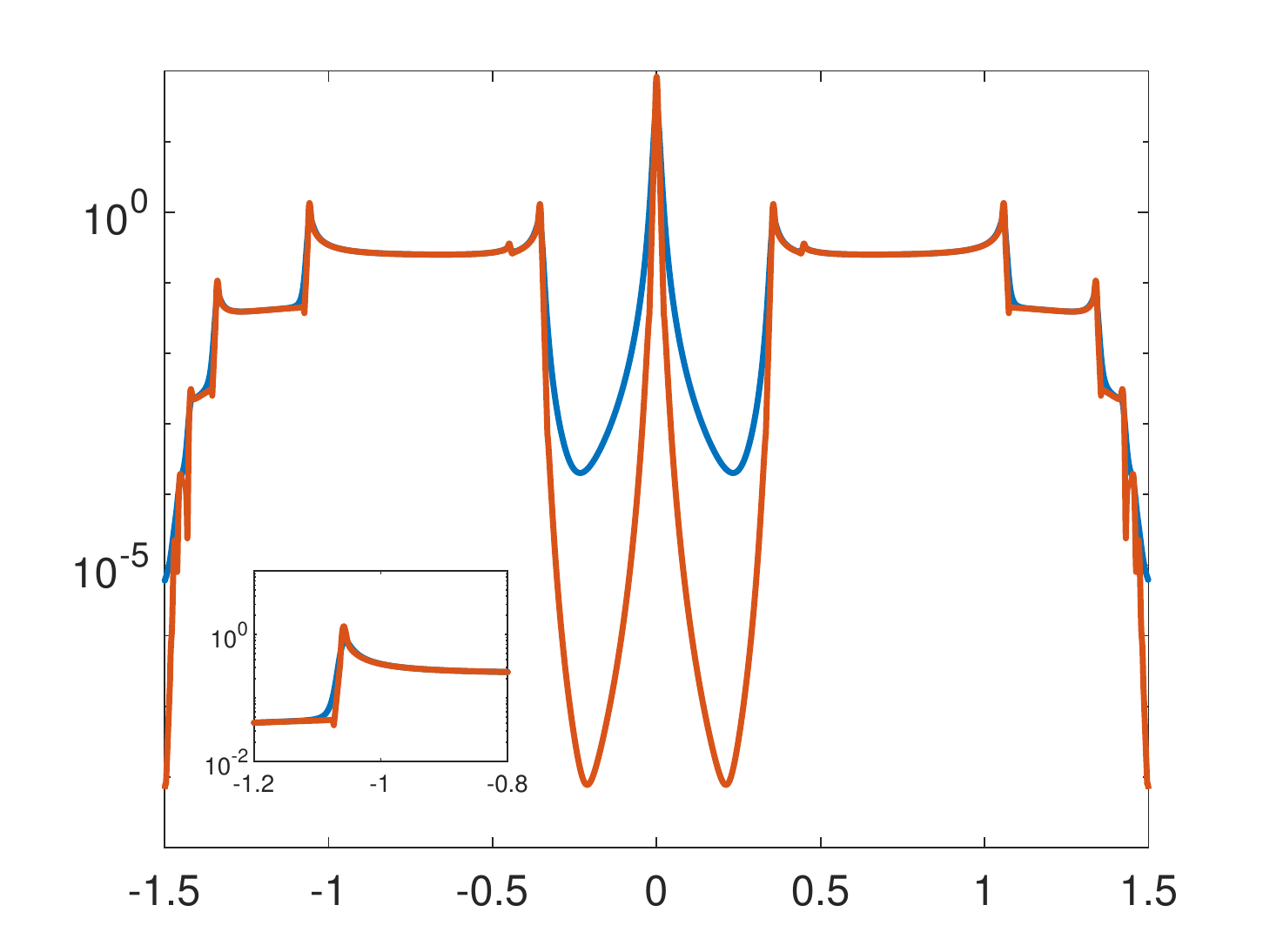}
		\put (45,73) {$\mu_{f}^{0.01}(x)$}
   \put (50,-1) {$x$}
	\put(58,65){\vector(-1,0){6}}
		\put(58,64){\small{}eigenvalue}
		\put(29,40){$\rho_f$}
		\put(30,44){\vector(0,11){10}}
		\put (71.5,21.5) { {$m=1$}}
	\put(71.5,23.5)  {\vector(-1,1){14}}
	\put (71.5,11.5) { {$m=6$}}
	\put(71.5,13.5)  {\vector(-1,0){14}}
   \end{overpic}
 \end{minipage}
	\begin{minipage}[b]{0.49\textwidth}
  \begin{overpic}[width=\textwidth,trim={0mm 0mm 0mm 0mm},clip]{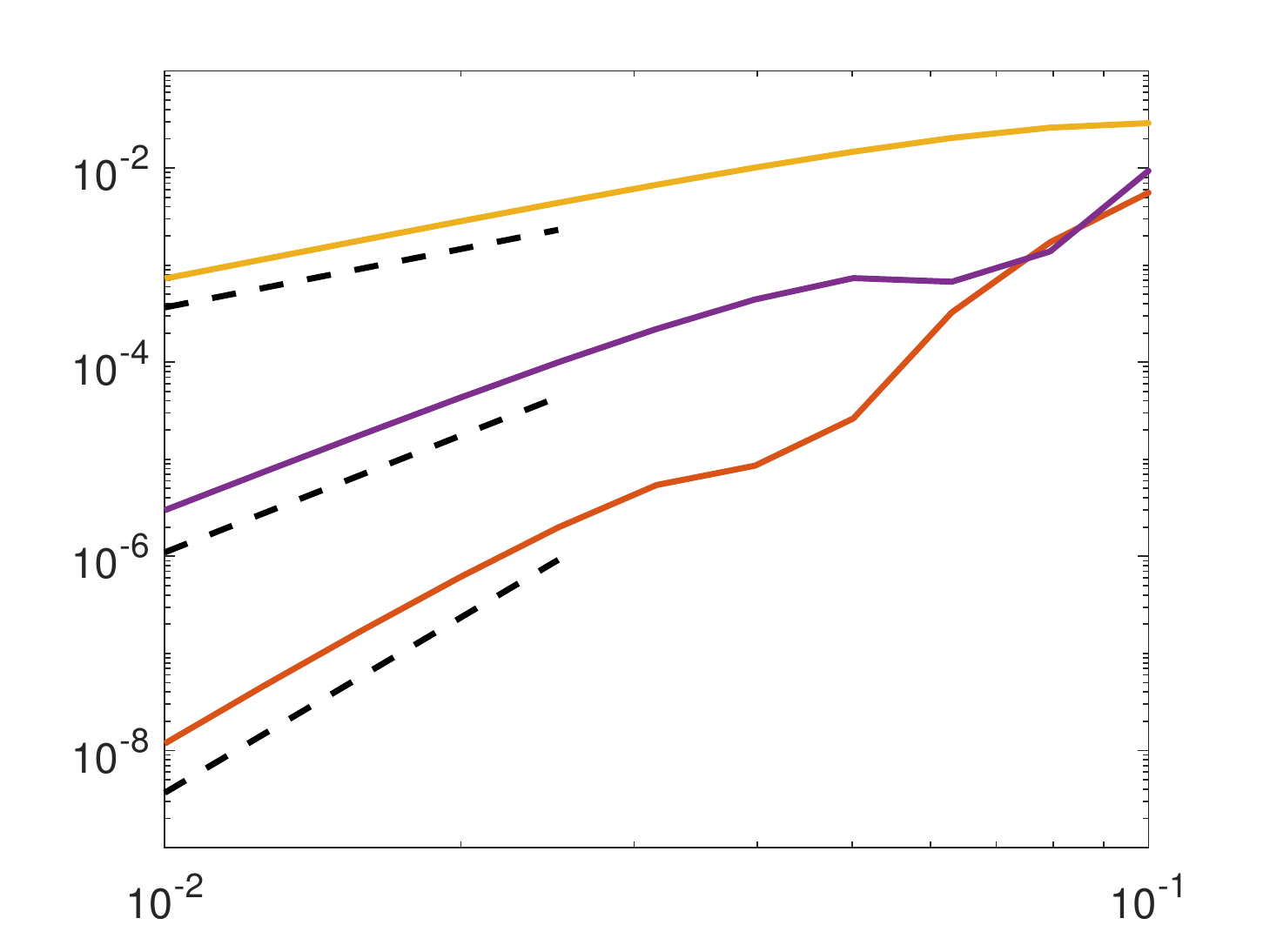}
	\put (13,73) {$|\mu_{f}^{\epsilon}(-0.9)-\rho_f(-0.9)|/\rho_f(-0.9)$}
   \put (50,-1) {$\epsilon$}
	\put (14,56) {\rotatebox{11} {$\displaystyle m=2$}}
	\put (14,37) {\rotatebox{23} {$\displaystyle m=4$}}
	\put (14,20) {\rotatebox{30} {$\displaystyle m=6$}}
   \end{overpic}
 \end{minipage}
	  \caption{Left: Smoothed spectral measures, $\smash{\mu_{f}^{0.01}}$, computed using the 1st and 6th order kernels. The zoomed-in section shows better resolution of jump discontinuities in $\rho_f$ for larger $m$. Right: Relative pointwise convergence to $\rho_f$ and expected rates shown as dashed black lines.}
\label{mcolb_fig:fourier}
\end{figure}

\vspace{1mm}

\noindent\textbf{Elliptic differential operator preconditioners:} A common use of $\mathcal{L}$ in \eqref{mcolb_operL} is preconditioning, where $\mathcal{B}$ is a preconditioner of $\mathcal{A}$ \cite{malek2014preconditioning}. For example, sometimes one can prove mesh-independent bounds on condition numbers for methods such as finite elements \cite{mardal2011preconditioning}, which are useful for applying Krylov space methods. The papers \cite{gergelits2019laplacian,gergelits2020generalized} discuss the spectrum of $\mathcal{L}$ in this context. The spectral measure of $\mathcal{L}$ and its discretizations determine the behavior of Krylov subspace methods. See \cite[Section 2]{gergelits2019laplacian} for a instructive example for which the spectrum is not enough.

\begin{figure}[!tbp]
 \centering
 \begin{minipage}[b]{0.5\textwidth}
  \begin{overpic}[width=\textwidth,trim={0mm 0mm 0mm 0mm},clip]{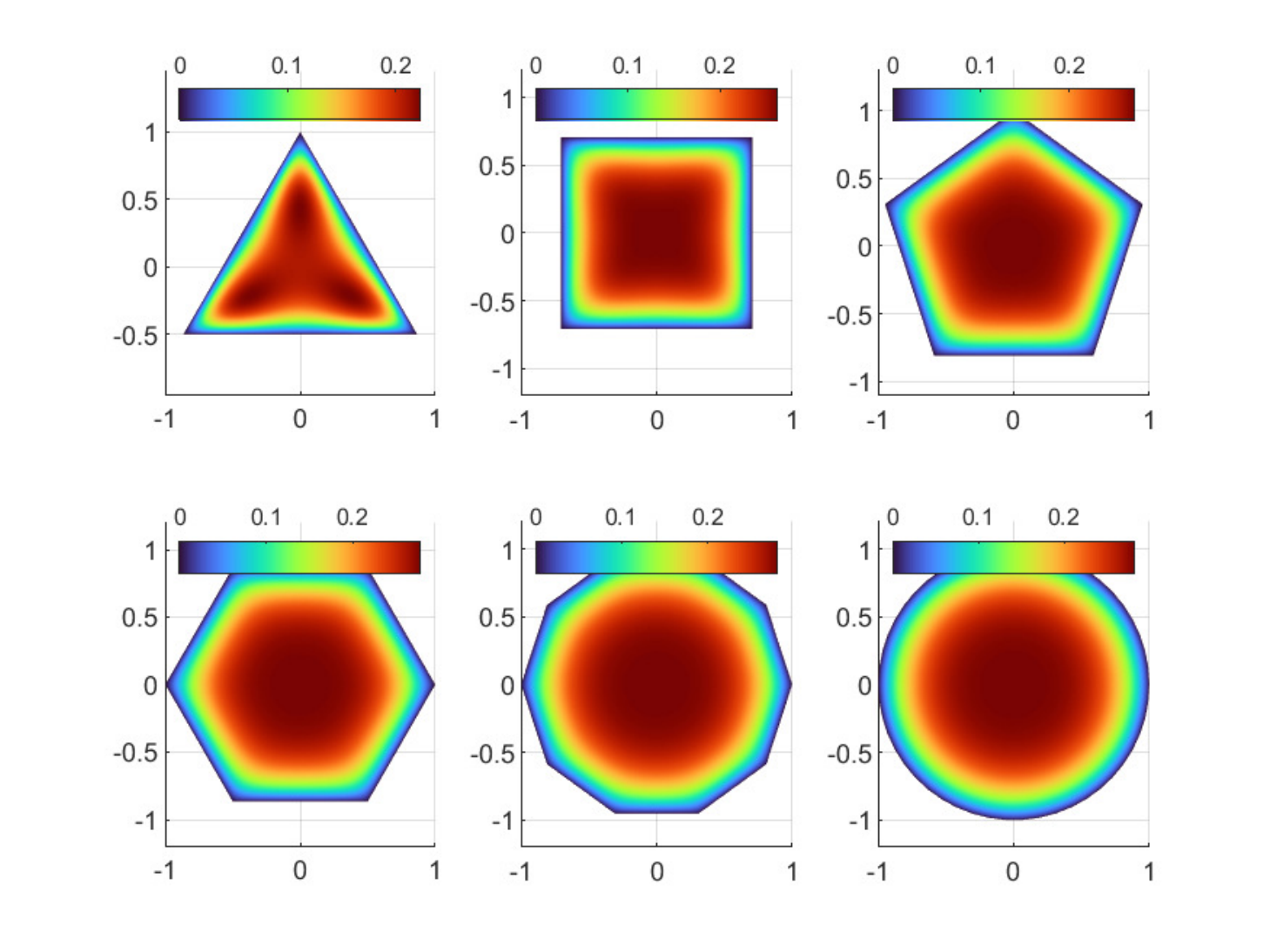}
		\put (12,73) {\small$n=3$}
		\put (40,73) {\small$n=4$}
		\put (69,73) {\small$n=5$}
		
		\put (12,36) {\small$n=6$}
		\put (40,36) {\small$n=10$}
		\put (69,36) {\small$n=\infty$}
      \end{overpic}
 \end{minipage}
	\begin{minipage}[b]{0.49\textwidth}
  \begin{overpic}[width=\textwidth,trim={0mm 0mm 0mm 0mm},clip]{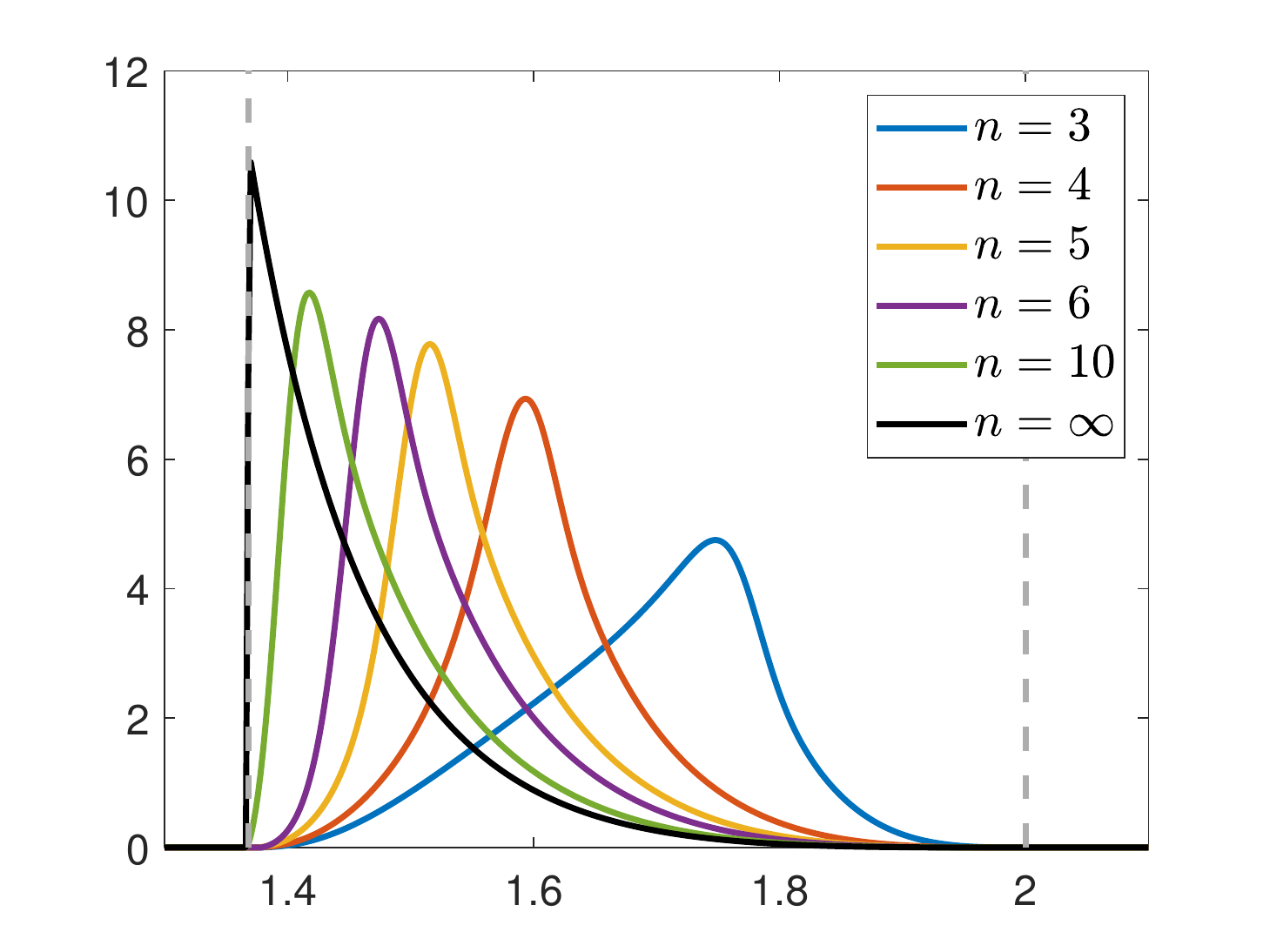}
\put (45,73) {$\mu_{f}^{\epsilon}(x)$}
   \put (50,-1) {$x$}
   \end{overpic}
 \end{minipage}
	  \caption{Left: Functions $f$ for different $n$. Right: Smoothed spectral measures, $\smash{\mu_{f}^{\epsilon}}$, for different $n$-gons computed using the 6th order kernel. For $n<\infty$ we use $\epsilon=0.05$ and for the circle we use $\epsilon=0.001$. The dashed vertical lines are the endpoints of the spectrum.}
\label{mcolb_fig:ultraSEM}
\end{figure}
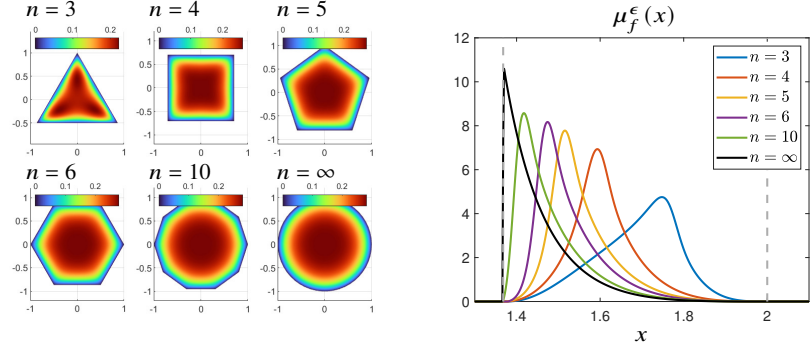

We follow \cite{dirichlet_precond} and consider a bounded Lipschitz domain $\Omega\subset\mathbb{R}^2$. We take
$$\setlength\abovedisplayskip{6pt}\setlength\belowdisplayskip{6pt}
\mathcal{A}u=-\nabla\cdot[(1+\exp(-x^2-y^2))\nabla u],\quad \mathcal{B}u=-\nabla^2 u,
$$
both with zero Dirichlet boundary conditions. The spectrum of $\mathcal{L}$ is the interval \cite{dirichlet_precond} $
\Lambda(\mathcal{L})=[\inf_{(x,y)\in\Omega}1+\exp(-x^2-y^2),\sup_{(x,y)\in\Omega}1+\exp(-x^2-y^2)\big],$ but the spectral measure is unknown. To solve the linear systems in \Cref{mcolb_alg:spec_meas2}, we use the ($hp$-adaptive and sparse) ultraspherical spectral element method \cite{ultraSEM}.

We take $\Omega$ to be a regular $n$-gon and set $f=C(\Omega)\mathcal{B}^{-1}g$, where $g(x,y)=x^2+y^2$ and $C(\Omega)$ are normalization constants so that each $\mu_f$ is a probability measure. \Cref{mcolb_fig:ultraSEM} (left) shows these $f$ and \cref{mcolb_fig:ultraSEM} (right) shows the smoothed spectral measures. The endpoints of the spectrum are shown as vertical dashed lines. The measures appear to be absolutely continuous and converge to the corresponding measure for the disk ($n=\infty$) as $n$ gets larger. To deal with the disk, we use separation of variables and solve the resulting radial ODEs using the ultraspherical spectral method \cite{Olver2013}.

\begin{acknowledgement}
MJC is supported by a Research Fellowship at Trinity College, Cambridge, and a Fondation Sciences Mathématiques de Paris Postdoctoral Fellowship at École Normale Supérieure. We thank Alex Townsend for pointing out that separation of variables efficiently deals with the $n=\infty$ case in \cref{mcolb_fig:ultraSEM} and for reading a draft version of the article. We thank Zdenek Strakos for discussions on the preconditioner example and for reading a draft version of the article.
\end{acknowledgement}

\end{document}